\documentclass[leqno,12pt]{amsart}
\usepackage{geometry, amsmath, amssymb, amsthm, mathrsfs, setspace}
\usepackage[dvipdfmx]{color} 
\usepackage[dvipdfmx]{graphicx}
\usepackage[parfill]{parskip}

\theoremstyle{plain}
\newtheorem{theorem}{Theorem}[section]
\newtheorem{lemma}[theorem]{Lemma}
\newtheorem{corollary}[theorem]{Corollary}

\newtheorem{definition}[theorem]{Definition}

\newtheorem*{acknowledge}{Aknowledgments}
\theoremstyle{definition}
\newtheorem{remark}[theorem]{Remark}

\newcommand{\del}{\partial}

\newcommand{\Z}{\mathbb{Z}}

\newcommand{\C}{\mathbb{C}}

\newcommand{\M}{\mathcal{M}}
\newcommand{\Es}{\mathfrak{S}}

\newcommand{\rot}{{\rm{rot}}}
\newcommand{\taub}{\overline{\tau}}

\makeatletter
\@addtoreset{equation}{section}

\makeatother

\begin{document}

\title[Compact Stein surfaces as branched covers]{Compact Stein surfaces as branched covers with same branch sets}

\author[Takahiro Oba]{Takahiro Oba}
\address{Department of Mathematics, Tokyo Institute of Technology, 2-12-1 Ookayama, Meguroku, Tokyo 152-8551, Japan}
\email{oba.t.ac@m.titech.ac.jp}

\begin{abstract}
	Loi and Piergallini showed that a smooth compact, connected $4$-manifold $X$ with boundary admits a Stein structure if and only if 
	$X$ is a simple branched cover of a $4$-disk $D^4$ branched along a positive braided surface $S$ in a bidisk $D_{1}^{2} \times D_{2}^{2} \approx D^4$. 
	For each integer $N \geq 2$, we construct a braided surface $S_{N}$ in $D^4$ 
	and simple branched covers $X_{N,1}, X_{N,2}, \dots , X_{N,N}$ of $D^{4}$ branched along $S_{N}$ such that 
	the covers have the same degrees, and they are mutually diffeomorphic, but the Stein structures associated to the covers are mutually not homotopic.
	Furthermore, by reinterpreting this result in terms of contact topology, 
	for each integer $N \geq 2$,
	we also construct a transverse link $L_{N}$ in the standard contact $3$-sphere $(S^3, \xi_{std})$
	and simple branched covers $M_{N,1}, M_{N,2}, \ldots, M_{N, N}$ of $S^3$ branched along $L_{N}$ 
	such that the covers have the same degrees, and they are mutually diffeomorphic, 
	but the contact structures associated to the covers are mutually not isotopic.
\end{abstract}


\date{\today}

\maketitle

   \section{Introduction.}
   
   \textit{Compact Stein surfaces} are sublevel sets of  exhausting strictly plurisubharmonic functions on Stein manifolds. 
   	They have been studied by using complex and symplectic geometry. 
	For example, Eliashberg \cite{El} characterized handle decompositions of compact Stein surfaces, 
	and Gompf \cite{Go} gave how to draw Kirby diagrams of them. 
	Since early 2000s, compact Stein surfaces also have been examined by using combinatorial techniques, and 
	research on them has been dramatically altered. 
	This development was caused by results of Loi and Piergallini \cite{LP} and Akbulut and Ozbagci \cite{AO}. 
	They showed that a smooth, oriented, connected, compact $4$-manifold $X$ with boundary admits a Stein structure $J$ if and only if 
	$X$ admits a \textit{positive allowable Lefschetz fibration} $f: X \rightarrow D^2$ (see Section \ref{section: LF}).
	It is known that Lefschetz fibrations are studied through mapping class groups, 
	so group theoretical approaches of them help us to treat compact Stein surfaces.
	For example, by using such techniques, 
	uniqueness results for Stein fillings of contact $3$-manifolds were proven in \cite{PV, Kal, KL, Ob}.
	For more various results, we refer the reader to \cite{Oz} as a survey on this subject.
	
	Loi and Piergallini also showed that a smooth, oriented, connected, compact 
	$4$-manifold $X$ with boundary admits a Stein structure $J$ if and only if 
	$X$ is a simple branched cover of a $4$-disk $D^4$ branched along a \textit{positive braided surface} $S$ in a bidisk $D_{1}^{2} \times D_{2}^{2}$ 
	(see Definition \ref{def: braided surfaces} and \ref{def: positive braided surfaces}),
	where, by rounding the corner of $D_{1}^{2} \times D_{2}^{2}$, it is identified with $D^{4}$.
	Unfortunately, although the fact is well-known, little is known about how Stein structures behave towards positive braided surfaces. 
	We can describe braided surfaces by using combinatorial tools 
	such as chart descriptions, quandles, and braid monodromies (cf. \cite{Ka}). 
	In order to use them effectively for research on compact Stein surfaces, 
	we need to better understand interactions between Stein structures and braided surfaces.
	
	In this paper, we consider 	whether or not, for a given positive braided surface $S$, 
	there exist more than one compact Stein surfaces as covers of $D^{4}$ branched along $S$ 
	such that the covers have the same degrees, 
	and they are mutually diffeomorphic but admit mutually distinct Stein structures. 
	The following theorem is a positive answer to this problem.
	
	\begin{theorem}\label{thm: main}
            	For a given integer $N \geq 2$, there exist a positive braided surface $S_{N}$ 
            	and simple branched covers $X_{N,1}, X_{N,2}, \dots, X_{N,N}$ of $D^4$ branched along $S_{N}$ such that
            	\begin{enumerate}
			\item the degrees of these covers are same, 
                    	\item $X_{N,1}, X_{N,2}, \dots, X_{N,N}$ are mutually diffeomorphic, and 
                    	\item Stein structures $J_{N,1}, J_{N,2}, \dots, J_{N,N}$ on $X_{N,1}, X_{N,2}, \dots, X_{N,N}$ respectively, 
                    	which are associated to the covers, are mutually not homotopic. 
            	\end{enumerate}
         \end{theorem}
	In the above theorem, 
	we consider as a Stein structure on the branched cover 
	one given by a Lefschetz fibration associated to the branched covering 
	(see Remark \ref{remark: Stein structures}). 
	 
	This theorem become more interesting, compared with 
	the case of 
	branched covers of $\mathbb{CP}^{2}$ and 
	cuspidal curves in $\mathbb{CP}^{2}$. 
	Here, a cuspidal curve is a projective plane curve 
	whose singular points are ordinary nodes and ordinary cusps.
	Chisini's conjecture (see {\cite{Chis}}) claims that 
	if $S \subset \mathbb{CP}^2$ is a cuspidal curve, 
	a \textit{generic} branched covering of $\mathbb{CP}^2$
	whose branch set is $S$ and degree is at least $5$ is unique up to covering isomorphism.
	Kulikov {\cite{Kul, Kul2}} showed that this conjecture is true under certain conditions. 
	The degree of each simple branched covering we will constructed 
	in the proof of Theorem \ref{thm: main} is $3N-1$ for each $N \geq 2$. 
	In addition, according to Rudolph \cite{Ru2}, 
	a positive braided surface is isotopic to the intersection of a complex analytic 
	curve with $D^4 \subset \C^2$, 
	and the converse is also true.
	Hence, an analogue of Chisini's conjecture does not hold for simple branched coverings of $D^4$ whose branch sets are 
	the intersections of complex analytic curves with $D^4$.

	We can reinterpret Theorem \ref{thm: main} in terms of contact $3$-manifolds and transverse links. 
	Let $M$ be an oriented, connected, closed $3$-manifold. 
	A $2$-plane field $\xi$ on $M$ is called a \textit{contact structure} on $M$ if there exists a $1$-form on $M$ such that 
	$\xi = \textrm{Ker} (\alpha)$ and $\alpha \wedge d\alpha >0$ with respect to the orientation of $M$, 
	and the pair $(M, \xi)$ is called a \textit{contact manifold}.
	An oriented link $L$ in $(M, \xi)$ is called a \textit{transverse link} if 
	$L$ is transverse to the contact plane $\xi_{x}$ at any point  $x$ in $L$. 
	Write $(D^2, id)$ for a \textit{supporting open book decomposition} of the standard contact $3$-sphere $(S^3, \xi_{std})$
	(cf. \cite{Et} for instance).
	Bennequin \cite{Be} showed that any transverse link in $(S^3, \xi_{std})$ can be braided about the binding of $(D^2, id)$.
	
	\begin{corollary}\label{cor: transverse}
            	For a given integer $N \geq 2$, there exist a transverse link $L_{N}$ in $(S^3, \xi_{std})$
            	and simple branched covers $M_{N, 1}, M_{N, 2}, \dots, M_{N, N}$ of $S^3$ branched along $L_{N}$ such that
                        	\begin{enumerate}
				\item the degrees of these covers are same, 
                        		\item $M_{N,1}, M_{N,2}, \dots, M_{N,N}$ are mutually diffeomorphic, and 
                        		\item contact structures $\xi_{N,1}, \xi_{N,2}, \dots, \xi_{N,N}$ on $M_{N,1}, M_{N,2}, \dots, M_{N,N}$ respectively, 
                        		which are associated to the covers, are mutually not isotopic. 
                        	\end{enumerate}	
	\end{corollary}
	Here, a contact structure on a branched cover means one supported 
	by an open book associated to the branched covering. 
	
	This article is organized as follows: 
	In Section $2$, we review some definitions and properties of mapping class groups, 
	braided surfaces, positive Lefschetz fibrations and 
	supporting open book decompositions.
	In Section $3$, 	first, we observe braids satisfying a certain condition, 
	called liftable braids, and, by using this notion, 
	prove a lemma to construct branched covers of $D^4$ in the proof of Theorem \ref{thm: main}. 
	Next, we review how to evaluate $\langle c_{1}(X,J), \cdot \, \rangle$, where $c_{1}(X, J)$ is 
	the first Chern class of a compact Stein surface $(X, J)$.
	Finally, we prove Theorem \ref{thm: main} by using contact structures, PALFs and Kirby diagrams coupled with the above lemma.
	
   	Throughout this article we will work in the smooth category.
   	We assume that the reader is familiar with basics of Kirby diagrams (see \cite[Chapter $4$, $5$]{GS}). 
	
    	\begin{acknowledge}
                \rm{The author would like to express his gratitude to Professor Hisaaki Endo 
            	for his encouragement and helpful comments for the draft of this article. 
		He would also like to thank Burak Ozbagci for his helpful comments and fruitful discussions. 
            	The author was partially supported by JSPS KAKENHI Grant Number 15J05214.}
   	 \end{acknowledge}

   \section{Preliminaries.}

\subsection{Mapping class groups}\label{section: mcg}

 	Let $\Sigma_{g, r}^{k}$ be an oriented, connected genus $g$ surface with $k$ marked points and $r$ boundary components. 
	We denote the mapping class group of $\Sigma_{g,r}^{k}$ by $\M_{g,r}^{k}$.
	More precisely, $\M_{g,r}^{k}$ is the group of isotopy classes of orientation preserving 
	self-diffeomorphisms of $\Sigma_{g,r}^{k}$
	which fix the marked points setwise and the boundary pointwise. 
	We also use the notations $\M _{g, r}$ if $k=0$, and 
	$\M_{\Sigma_{g, r}^{k}}$ for $\M_{g,r}^{k}$.
	For a simple closed curve $\alpha$ in $\Sigma_{g, r}^{k}$, $t_{\alpha} \in \M_{g, r}^{k}$ denotes 
	the \textit{right-handed Dehn twist} along $\alpha$.
	Furthermore, for a simple arc $a$ connecting two distinct marked points in $\Sigma_{g,r}^{k}$, 
	write $\tau_{a} \in \M_{g, r}^{k}$ for 
	the \textit{right-handed half-twist} along $a$.
	We will use the opposite notation to the usual functional one for the products in $\M_{g, r}^{k}$, 
	i.e. $h_{1} h_{2} $ means that we apply $h_{1}$ first and then $h_{2}$. 
	Moreover, for a subset $A \subset \Sigma_{g,r}^{k}$ and $h \in \M_{g,r}^{k}$, 
	the notation $(A)h$ means the image of $A$ under $h$.
	
	It is well known that the braid group $B_{m}$ on $m$ strands 
	can be identified with the mapping class group $\M_{0,1}^{m}$ as follows
	(cf. \cite[Section 3.2]{EV}): 
	Consider an $m$-marked disk $\Sigma_{0,1}^{m}$ 
	as the unit closed disk $\mathbb{D}_{m} \subset \C$ with $m$ marked points which lie on the real axis.
	Set $P_{1}, P_{2}, \cdots, P_{m}$ as the $m$ marked points, 
	where $P_{1} < P_{2} < \cdots < P_{m}$.
	Define an arc $A_{i}$ on the real axis to be one with end points in $P_{i}$ and $P_{i+1}$.
	Then, the $i$-th standard generator $\sigma_{i}$ of $B_{m}$ 
	can be identified with the right-handed half-twist $\tau_{A_{i}}$. 
	In this article, under this identification, 
	a simple arc with end points in the set of marked points 
	represents the corresponding element of $B_{m}$ to the half-twist along the arc.	 
   
   \subsection{Braided surfaces.}\label{section: braided surfaces}
   
    	Let $D^2_{1}$ and $D^{2}_{2}$ be oriented $2$-disks.
    
    	\begin{definition}\label{def: braided surfaces}
                	A properly embedded surface $S$ in  $D^{2}_{1} \times D^{2}_{2}$ is called a (simply) \textit{braided surface} of degree $m$ 
                	if the first projection $pr_{1}: D^{2}_{1} \times D^{2}_{2} \rightarrow D^{2}_1$ restricts 
            	to a simple branched covering $p_{S} := pr_{1} | S : S \rightarrow D^{2}_{1}$ of degree $m$.
    	\end{definition}
    
    	We will review briefly braid monodromies of braided surfaces 
    	(see, for more details, \cite[Section 3]{APZ}, \cite[Chapter 16, 17]{Ka}, \cite[\S 1, 2]{Ru}).
    	Before that, we recall a special basis for the fundamental group of a punctured disk.
    	Let $Q$ be a set of $n$ points $x_{1}, x_{2}, \dots, x_{n}$ in the interior of an oriented $2$-disk $D^{2}$ with the standard orientation and 
	let $x_{0}$ be a point in $\del D^{2}$. 
	Since the fundamental group $\pi_{1}(D^{2} - Q, x_{0})$ is a free group of rank $n$, 
	we give a basis for this group as follows: 
	Take a collection of oriented paths $s_{1}, s_{2}, \dots, s_{n}$ 
	starting from $x_{0}$ to each $x_{i}$, respectively. Assume that $s_{i}$ and $s_{j}$, if $i \neq j$, are disjoint except $x_{0}$, 
	and the arcs $s_{1}, s_{2}, \dots, s_{n}$ are indexed so that they appear in order as we move counterclockwise about $x_{0}$. 
	By using the path $s_{i}$, 
	connect $x_{0}$ to a small oriented disk around each $x_{i}$ with the same orientation of $D^2$. 
	Then, we obtain an oriented loop $\gamma_{i}$ based at $x_{0}$, 
	and $ \gamma_{1}, \gamma_{2}, \dots , \gamma_{n} $ freely generate $\pi_{1}(D^{2}-Q, x_{0})$.  
    	The ordered $n$-tuple $(\gamma_{1}, \gamma_{2}, \dots, \gamma_{n})$ is called a \textit{Hurwitz system} 
    	for $(Q, x_{0})$ (see Figure \ref{fig: Hurwitz}).
    	
	\begin{figure}[ht]
		\begin{center}
			\includegraphics[width=130pt]{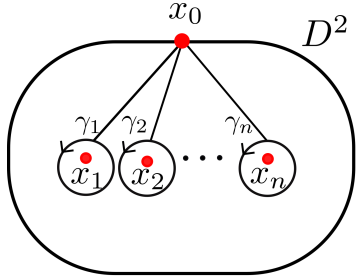}
			\caption{The standard Hurwitz system for $(Q, x_{0})$.}
			\label{fig: Hurwitz}
		\end{center}
	\end{figure}
	
	We now turn to the case of braided surfaces.
	Let $Q(p_{S}): = \{ a_{1}, a_{2}, \dots , a_{n} \} \subset \rm{Int}\, D_{1}^2$ be 
	the set of branch points of the branched covering $p_{S} : S \rightarrow D^{2}_{1}$. 
	Fix a point $a_{0}$ in $\del D^{2}_{1}$ and Hurwitz system $(\gamma_{1}, \gamma_{2}, \dots, \gamma_{n})$ for $(Q(p_{S}), a_{0})$.
	For each $\gamma_{i}$, the restriction of $pr_{1}$ to $pr_{1}^{-1} (\gamma_{i})$ induces a trivial disk bundle over $\gamma_{i}$.
	Since for any point $a \in \gamma_{i} \subset D_{1}^{2}- Q(p_{S})$, $p_{S}^{-1}(a)$ consists of $m$ points,
	each fiber $pr_{1}^{-1} (a) = \{ a \} \times D_{2}^{2} =: D_{2}^{2}(a)$ of the disk bundle contains $m$ points, 
	which are the intersection points of $D_{2}^{2}(a_{0})$ and $S$.
	Hence, we associate an element $\beta_{i} \in B_{m}$ 
	to $\gamma_{i}$ as a motion of the set $D_{2}^{2} (a_{0}) \cap S$ over $\gamma_{i}$.
	By this correspondence, we can define a homomorphism $\omega_{S}: \pi_{1}(D_{1}^{2}-Q(p_{S})) \rightarrow B_{m}$ by 
	$\omega_{S} (\gamma_{i}) = \beta_{i}$ for each $i$. 
	This homomorphism $\omega_{S}$ is called a \textit{braid monodromy} of $S$.
	The ordered $n$-tuple $(\omega_{S}(\gamma_{1}), \omega_{S}(\gamma_{2}), \dots, \omega_{S}(\gamma_{n}))$ 
	is also called a braid monodromy of $S$.
	Since $p_{S}$ is a simple branched covering, 
	each $\omega_{S} (\gamma_{i})$ is a conjugate $w_{j}^{-1} \sigma_{j_{i}}^{\varepsilon_{i}} w_{j}$ of $\sigma_{j_{i}}^{\varepsilon_{i}}$ 
	for some $w_{j} \in B_{m}$ and $\varepsilon_{i} \in \{\pm 1 \}$. 
	It is known that, for a finite set $Q$ and representation $\omega: \pi_{1}(D^2_{1} - Q, a_{0}) \rightarrow B_{m}$ as above, 
	we can construct a braided surface of degree $m$ whose branch set is $Q$ and braid monodromy is $\omega$.
	Obviously, since $p_S$ is a branched covering, we consider a \textit{covering monodromy} of $p_S$, 
	i.e. a representation $\rho_{S} : \pi_{1}(D^{2}_{1}- Q(p_{S}) , a_{0}) \rightarrow \Es_{m}$, 
	where $\Es_{m}$ is the symmetric permutation group of degree $m$.
	Note that each $\rho_{S} (\gamma_{i}) \in \Es_{m}$ is a transposition because $p_S $ is simple. 	
	Furthermore, we also remark that $\omega_{S}$ is a lift of $\rho_{S}$ to $B_{m}$.
		
	At the end of this subsection, 
	we define a crucial notion to examine compact Stein surfaces by braided surfaces.
	
	\begin{definition}\label{def: positive braided surfaces}
            	A braided surface $S$ is called \textit{positive} if each $\omega_{S} (\gamma_{i}) $ is positive, that is, 
            	for a braid monodromy 
            	$(w_{1}^{-1} \sigma^{\varepsilon_{1}}_{j_{1}} w_{1}, w_{2}^{-1} \sigma^{\varepsilon_{2}}_{j_{2}} w_{2}, $ 
		$\dots, w_{n}^{-1} \sigma^{\varepsilon_{n}}_{j_{n}} w_{n})$ of $S$, 
            	each $\varepsilon_{i}$ is $1$.
	\end{definition}

\subsection{Lefschetz fibrations and simple branched coverings. }\label{section: LF}

	We will briefly review positive Lefschetz fibrations and their monodromies (see \cite[Chapter 8]{GS}).
	
	Let $X$ be an oriented, connected, compact $4$-manifold.
	
   	\begin{definition}\label{def: LF}
		A smooth map $f:X\rightarrow D^2$ is called a \textit{positive Lefschetz fibration} if there exists the set $Q (f)$ of points
		$a_{1}, a_{2}, \dots , a_{n}$ of the interior of $D^2$ such that
	
		\begin{enumerate}
			\item $f|f^{-1}(D^2-Q (f) ):f^{-1}(D^2 -Q (f) ) \rightarrow D^2- Q (f)$
				is a smooth fiber bundle over $D^2- Q (f)$ with fiber diffeomorphic to an oriented compact surface $\Sigma$ with boundary,
			\item $ a_{1}, a_{2}, \dots , a_{n} $ are the critical values of $f$, and each singular fiber $f^{-1}(a_{i})$ has
				 a unique critical point $p_{i} \in f^{-1}( a_{i})$, and 
			\item for each $p_{i}$ and $a_{i}$, there are local complex coordinate charts with respect to the orientations of $X$ 
				and $D^2$ such that locally f can be written as $f(z_{1},z_{2}) = z_{1}^2 + z_{2}^2$.
		\end{enumerate}
  	 \end{definition}
   
   	A positive Lefschetz fibration $f:X\rightarrow D^2$ can be described by the mapping class group $\M_{\Sigma}$ of the fiber $\Sigma$ of $f$. 
	Let $a_{0} \in \del D^2$ be a fixed base point. 
	Take a Hurwitz system $(\gamma_{1}, \gamma_{2}, \dots , \gamma_{n})$ for $(Q (f), a_{0})$. 
	We can consider a homomorphism $\eta_{f}: \pi_{1}(D^2-Q (f), a_{0}) \rightarrow \M_{\Sigma}$ as follows: 
	The positive Lefschetz fibration $f$ restricts to 
	a fiber bundle $f| f^{-1}(\gamma_{i}): f^{-1}(\gamma_{i}) \rightarrow \gamma_{i}$ for each $\gamma_i$.
	The monodromy of this fiber bundle is the right-handed Dehn twist $t_{\alpha_{i}}$ along a simple closed curve $\alpha_{i}$ in $f^{-1}(a_{0}) \approx \Sigma$.
	The simple closed curve $\alpha_{i}$ is called a \textit{vanishing cycle} of the singular fiber $f^{-1}(a_{i})$.
	Define $\eta_{f}: \pi_{1}(D^2-Q (f), a_{0}) \rightarrow \M_{\Sigma}$ by $\eta_{f}(\gamma_{i}) = t_{\alpha_{i}}$ for each $\gamma_{i}$
	and call $\eta_{f}$ a \textit{monodromy} of $f$. 
	We also call the ordered $n$-tuple $(t_{\alpha_{1}}, t_{\alpha_{2}}, \dots, t_{\alpha_{n}})$ a monodromy of $f$.
	We say a positive Lefschetz fibration to be \textit{allowable} if all of the vanishing cycles $\alpha_{1}, \alpha_{2}, \dots, \alpha_{n}$
	are homologically non-trivial in the fiber. 
	After this, we call a positive allowable Lefschetz fibration a \textit{PALF} shortly.
   
   	The following theorem tells us that PALFs and positive braided surfaces 
	are so important to study compact Stein surfaces.
   
        \begin{theorem}[Loi and Piergallini {\cite[Theorem 3]{LP}} (cf. Akbulut and Ozbagci {\cite[Theorem 5]{AO}})]\label{thm: LP}
        Let $X$ be an oriented, connected, compact $4$-manifold with boundary.
        Then the following conditions are equivalent:
                	\begin{enumerate}
                    \item $X$ is a compact Stein surface, that is, $X$ admits a Stein structure $J$; 
                    \item $X$ admits a PALF $f: X \rightarrow D^2$; 
                    \item $X$ is a simple branched cover of $D^{4}$ branched along a positive braided surface in $D_{1}^{2} \times D_{2}^{2} \approx D^4$.
                \end{enumerate}
         \end{theorem}
         
	Note that according to \cite[Proposition 1, 2]{LP} and the proof of Theorem \ref{thm: LP}, 
	for a given PALF $f:X \rightarrow D^2$, 
	we can construct a simple branched covering $p: X \rightarrow D^4$ whose branch set is a positive braided surface $S$ so that 
	$f = pr_{1} \circ p$ and $Q(p_{S})=Q(f)$.
	Conversely, for a given simple branched covering $p: X \rightarrow D^4$ whose branch set is a positive braided surface $S$, 
	$f := pr_{1} \circ p : X \rightarrow D^{2}_{1}$ is a PALF, and $Q(f)=Q(p_S)$ (cf. Figure \ref{fig: BSandLF}).
	Suppose $a \in D_{1}^{2}$ is a regular point of the above PALF $f= pr_{1} \circ p$.
   	The point $a$ is also a regular point of $p_{S}$.
   	Since $p$ is a simple branched covering branched along $S$, 
   	$p$ restricts to a simple branched covering $p| p^{-1}( D_{2}^{2} (a) ) : p^{-1}(D_{2}^{2} (a)) \rightarrow D_{2}^{2} (a)$ 
   	whose branch set is $S \cap D_{2}^{2} (a)$. 
   	It is easy to check that $p^{-1}(D_{2}^{2} (a))$ is the regular fiber $f^{-1}(a)$ of $f$.
   
  	 \begin{remark}\label{remark: Stein structures}
               	It is known that the total space of a PALF admits a Stein structure by using the handle decomposition given by the PALF
            	 (see. \cite[Theorem 5]{AO}).
               	As mentioned above, a simple branched covering $p: X \rightarrow D^4 \approx D_{1}^{2} \times D_{2}^{2}$ 
            	branched along a positive braided surface gives 
               	a PALF $pr_{1} \circ p: X \rightarrow D_{1}^{2}$. 
               	Thus, we equip the cover $X$ with the Stein structure coming from the PALF $pr_{1} \circ p$. 
   	\end{remark}
   
   	\begin{figure}[ht]
		\begin{center}
			\includegraphics[width=250pt]{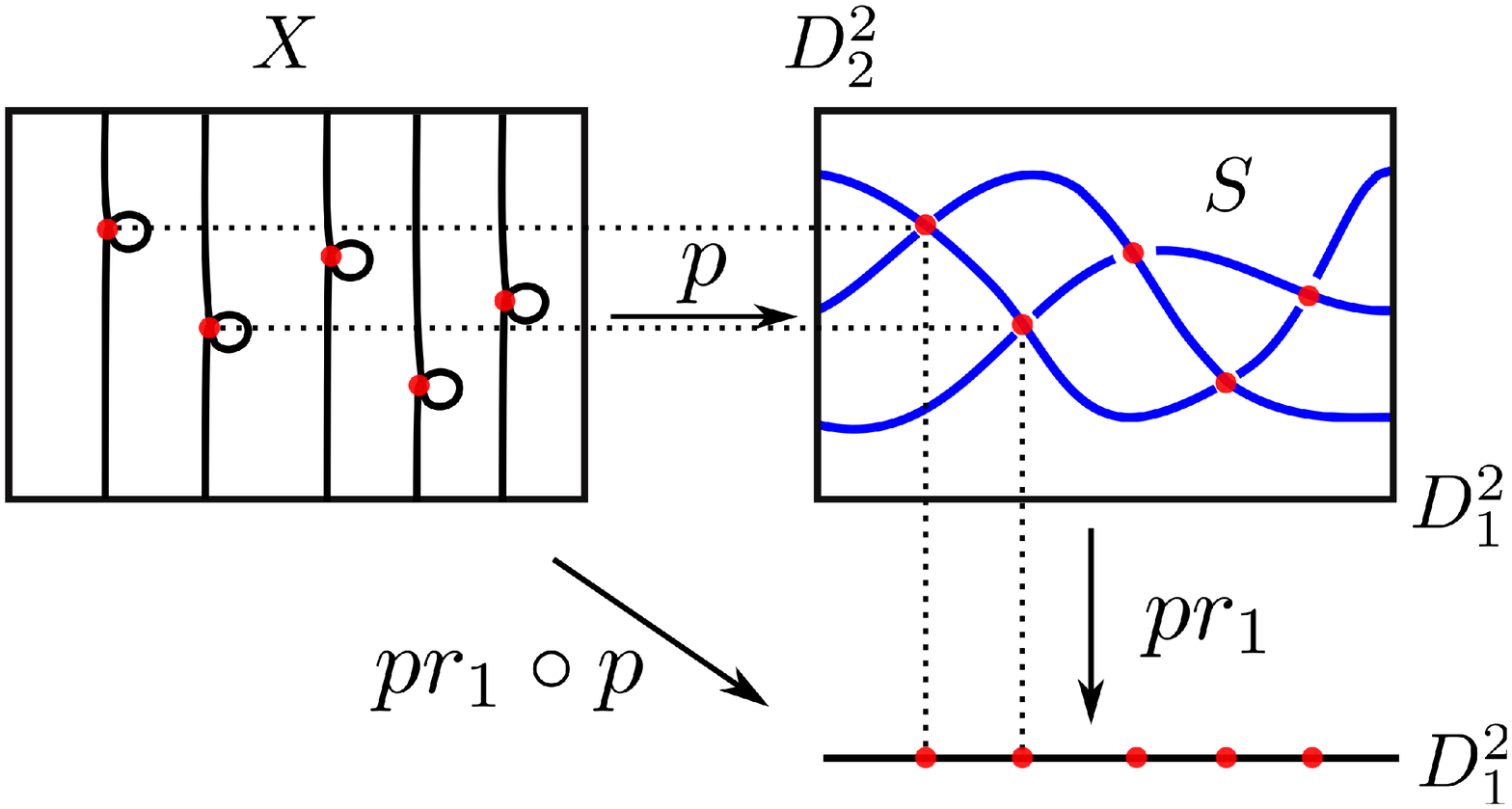}
			\caption{The left (resp. right) square represents the total space $X$ of $p$ (resp. $D_{1}^{2} \times D_{2}^{2}$).
			The red points in $X$ (resp. $D_{1}^{2} \times D_{2}^{2}$) represents the critical points of the PALF $pr_{1} \circ p$ 
			(resp. the branched covering $p_{S}$).}
			\label{fig: BSandLF}
		\end{center}
	\end{figure}

\subsection{Supporting open book decompositions of tight lens spaces $L(2N, 1)$}\label{section: OB}
		
	In order to show Theorem \ref{thm: main}, we will discuss contact structures on the lens space $L(2N, 1)$ via open books.
	Hence, we review contact structures on $L(2N, 1)$ and their supporting open book decompositions (see \cite{Et}, \cite[Section 2]{PV} for more details).
	
	To begin with, we recall a stabilization of a Legendrian knot.
	Let $L$ be a Legendrian knot in $(S^3, \xi_{std})$.
	A \textit{positive} (resp. \textit{negative}) \textit{stabilization} on $L$ is a Legendrian knot $L_{+}$ (resp. $L_{-}$) 
	obtained from adding a zig-zag to $L$ as depicted in the left (resp. right) of Figure \ref{fig: stabilization}.
	If $L$ lies on a page of a supporting open book decomposition of $(S^3, \xi_{std})$, 
	we stabilize the open book and modify $L$ as shown in the bottom of Figure \ref{fig: stabilization}.
	
	\begin{figure}[ht]
		\begin{center}
			\includegraphics[width=300pt]{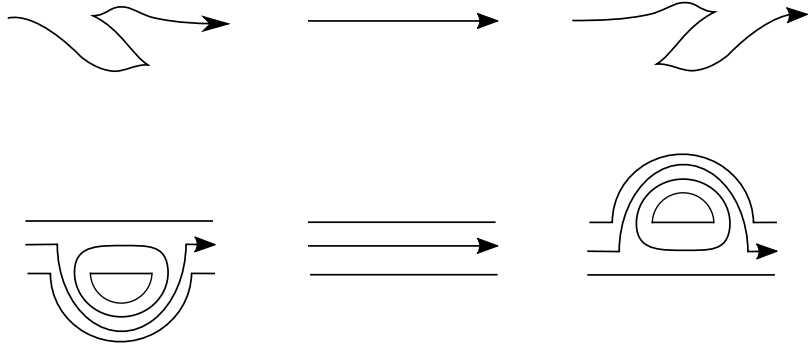}
			\caption{The stabilizations $L_{+}$ and $L_{-}$ of a Legendrian knot $L$ and the corresponding open books.}
			\label{fig: stabilization}
		\end{center}
	\end{figure}
	
	Let $(L(2N,1), \xi_{N,j})$ be a tight contact manifold obtained from the Legendrian surgery  on 
	the Legendrian knot $O_{N,j}$ shown in Figure \ref{fig: Legendrian knot}. 
	Since $O_{N,j}$ is a Legendrian knot with $tb= -2N+1$ and $\rot=2(N-j)$, 
	$O_{N,1}, O_{N,2}, \dots , O_{N,N}$ are mutually not Legendrian isotopic by \cite[THEOREM 1.1]{EF}.
	Thus, according to Honda's classification of tight contact structures on $L(2N, 1)$ (see \cite[Theorem 2.1]{Ho}), 
	$\xi_{N,1}, \xi_{N,2}, \dots, \xi_{N,N}$ are mutually not isotopic tight contact structures. 
	To obtain a supporting open book decomposition of $(L(2N,1), \xi_{N,j})$, 
	we explain how $O_{N,j}$ is obtained from the Legendrian unknot $O$ with $tb=-1$.
	Repeat $j-1$ times positively and negatively stabilizing $O$ alternately and,
	after that, perform $2(N-j)$ times negatively stabilizing the resulting Legendrian knot.
 	Hence, the corresponding open book decomposition of $(L(2N,1), \xi_{N,j})$ is given 
	as in Figure \ref{fig: L(2N,1)}. 
	We write $(\Sigma_{0, 2N}, \varphi_{N, j})$ for this open book, 
	where the monodromy $\varphi_{N,j}$ is given by 
	\[
		\varphi_{N,j} = t_{\alpha_{N,j}} t_{\beta_{N,j}} t_{\delta_{2}} t_{\delta_{3}} \cdots t_{\delta_{2N-1}}.
	\]
		
   	\begin{figure}[htb]
		\begin{center}
			\includegraphics[width=200pt]{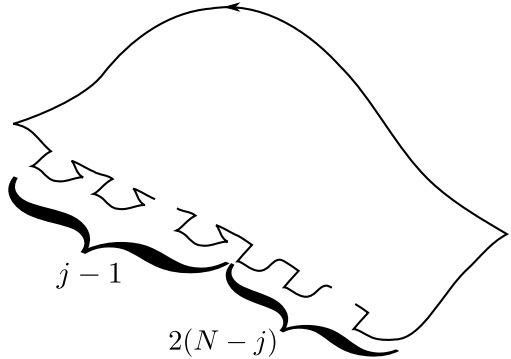}
			\caption{Legendrian knot $O_{N,j}$.}
			\label{fig: Legendrian knot}
		\end{center}
	\end{figure}

	\begin{figure}[htb]
		\begin{center}
			\includegraphics[width=350pt]{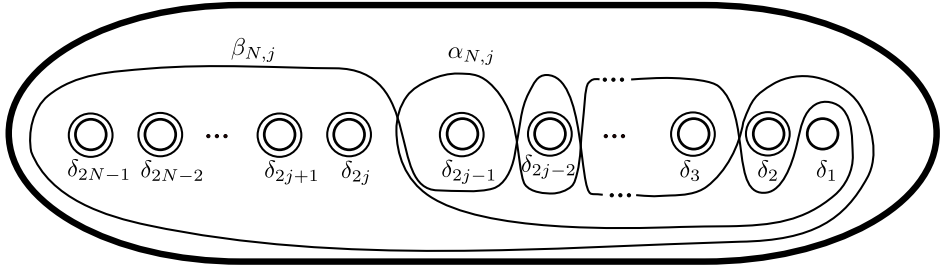}
			\caption{Supporting open book decomposition of $(L(2N,1), \xi_{N,j})$ .}
			\label{fig: L(2N,1)}
		\end{center}
	\end{figure}

  \section{Main Results.}\label{section: main}
   
        Let $\Sigma$ be an oriented, connected, compact surface with boundary.
        Suppose $q: \Sigma \rightarrow D^2$ is a simple branched covering of degree $d$.
        This covering $q$ determines a covering monodromy $\rho_{q}: \pi_{1} (D^2 - Q (q), b_{0}) \rightarrow \Es_{d}$.
        Identifying $B_{m}$ with $\M_{0,1}^{m}$ as in Section \ref{section: mcg}, 
        we associate with a given $\beta \in B_{m}$ the mapping class $h_{\beta} \in \M_{0,1}^{m}$.
        	We call $\beta \in B_{m}$ or $h_{\beta} \in \M_{0,1}^{m}$ 
	\textit{liftable} with respect to the branched covering $q: \Sigma \rightarrow D^2$ with $m$ branch points 
	if there exists an orientation preserving diffeomorphism $H_{\beta}$ of $\Sigma$ 
	such that $ q \circ H_{\beta} = h \circ q$ for some representative $h$ of $h_{\beta}$. 
	Note that, in this definition,  
	we identify with $\mathbb{D}_{m}$ the base disk $D^{2}$ with $m$ branch points 
	 and consider $\M_{0,1}^{m}$ as the mapping class group of $\mathbb{D}_{m}$. 
	In \cite[Lemma $4.3.3$]{MM}, it is shown that, if $h_{\beta} \in \M_{0,1}^{m}$ is liftable with respect to $q$, 
	then we have 
	\begin{align}\label{eq: liftability}
	\rho_{q} \circ {h_{\beta}}_{*} = \rho_{q}
	\end{align} 
	for 
	the induced isomorphism ${h_{\beta}}_{*}: \pi_{1}(D^2 - Q(q), \beta_{0}) \rightarrow \pi_{1}(D^2 - Q(q), b_{0})$.
	
    	The following lemma is useful to construct simple branched covers of $D^4$. 
    
    	\begin{lemma}\label{lemma: lift}
                	Let $S$ be a positive braided surface of degree $m$ with braid monodromy 
            	$(w_{1}^{-1}\sigma_{j_{1}}w_{1}, w_{2}^{-1}\sigma_{j_{2}}w_{2}, \dots, w_{n}^{-1}\sigma_{j_{n}}w_{n})$ and 
            	let $a_{0}$ be a fixed base point in $\del D_{1}^{2}$.
                	Suppose $q: \Sigma \rightarrow D_{2}^{2} (a_{0})$ is a simple branched covering of degree $d$ 
            	with branch set $S \cap D_{2}^{2} (a_{0})$  and covering monodromy $\rho_{q}$.
            	If each $w_{i}^{-1}\sigma_{j_{i}}w_{i} \in B_{m}$ is liftable with respect to $q$, 
                	then there exist an oriented, connected, compact $4$-manifold $X$ and 
            	a simple branched covering $p: X \rightarrow D^{4}$ branched along $S$ such that $p| p^{-1}(D_{2}^{2}(a_{0})) = q$.
   	\end{lemma}
	   
  	\begin{proof}
   		Fix a point $b_{0} \in \del D_{2}^2$. 
   		Let $( \gamma'_{1}, \gamma'_{2}, \dots , \gamma'_{m} )$ be the standard Hurwitz system 
		for $ (D_{2}^{2} (a_{0} - S), (a_{0}, b_{0}) )$ as in Figure \ref{fig: Hurwitz}.
   		It is known that 
   		\[
			\pi_{1} (D^4 - S, (a_{0}, b_{0})) = \langle \gamma'_{1}, \gamma'_{2}, \dots, \gamma'_{m}  | 
			 (\gamma'_{j_{i}}) {w_{i}}_{*} = (\gamma'_{j_{i}+1}) {w_{i}}_{*} \ i = 1, 2, \dots n \rangle , 
		\]
		where each ${w_{i}}_{*}$ is the Artin automorphism of the free group $\langle \gamma'_{1}, \gamma'_{2}, \dots \gamma'_{m}  \rangle$ 
		defined by 
		\[
			(\gamma'_{j}) {\sigma_{i}}_{*} =
				\begin{cases}
					\gamma'_{i} \gamma'_{i+1} {\gamma'}_{i}^{-1} &  (j=i), \\
					\gamma'_{i} & (j=i+1), \\
					\gamma'_{j} & (j \neq i, i+1).
				\end{cases}
		\]
            	More precisely we refer the reader to \cite[p.133]{Fo} and \cite[PROPOSITION $4.1$]{Ru} about this fundamental group.
            	If we show $\rho_{q} ( (\gamma'_{j_{i}}) {w_{i}}_{*}) = \rho_{q} ((\gamma'_{j_{i}+1}) {w_{i}}_{*})$ for each $i$, 
            	we conclude that $\rho_{q}$ induces a homomorphism $\rho: \pi_{1} (D^4 - S, (a_{0}, b_{0})) \rightarrow \Es_{d}$, and 
            	this $\rho$ determines a simple branched covering $p:X \rightarrow D^4$ of degree $d$ whose branch set is $S$.
            
            	For each $i$, we have 
            	\begin{eqnarray}
            		(\gamma'_{j_{i}}) {w_{i}}_{*}& =  & ((\gamma'_{j_{i}+1}) {\sigma_{j_{i}}}_{*} ) {w_{i}}_{*} \nonumber \\
            		& = & (\gamma'_{j_{i}+1}) \{ {w_{i}}_{*}  (w_{i})^{-1}_{*}\}  {\sigma_{j_{i}}}_{*}  {w_{i}}_{*} \nonumber \\
            		& = & ((\gamma'_{j_{i}+1}) {w_{i}}_{*} ) ( w_{i}^{-1}{\sigma_{j_{i}}}w_{i})_{*} .\nonumber 
            	\end{eqnarray}
            	Since each $w_{i}^{-1}\sigma_{j_{i}}w_{i}$ is liftable and the equation (\ref{eq: liftability}) holds for any liftable braid, 
            	\begin{eqnarray}
            		\rho_{q} ( (\gamma'_{j_{i}}) {w_{i}}_{*} ) & =  & \rho_{q} (((\gamma'_{j_{i}+1}) {w_{i}}_{*} ) (w_{i}^{-1}{\sigma_{j_{i}}}w_{i})_{*})  \nonumber \\
			& = & ( \rho_{q} \circ   (w_{i}^{-1} {\sigma_{j_{i}}}w_{i})_{*})  ((\gamma'_{j_{i}+1}) {w_{i}}_{*}  ) \nonumber \\
			& = & \rho_{q} ( (\gamma'_{j_{i} +1}) {w_{i}}_{*} ). \nonumber
		\end{eqnarray}
		According to the above construction of $p$, we can easily check $p | p^{-1} (D_{2}^{2} (a_{0})) = q$.
	
   	\end{proof}

   	In the proof of Theorem \ref{thm: main}, we use the first Chern class of a compact Stein surface.
   	In order to compute them, 
	we make use of the following facts in \cite[Section 3]{EO} and \cite[Proposition 2.3]{Go}.
   	Let $f: X \rightarrow D^{2}$ denote a PALF with fiber $\Sigma$ and let 
	$\alpha_{1}, \alpha_{2}, \dots, \alpha_{n}$ be the its vanishing cycles.
	Note that $X$ admits a Stein structure $J$ by Theorem \ref{thm: LP}.
	Once choosing a trivialization of the regular fiber of $f$, 
	the rotation number $\rot (C)$ of a simple closed curve $C$ is defined with respect to this trivialization.
	Give an orientation to a vanishing cycle $\alpha_{i}$ and regard this $\alpha_{i}$ 
	as a generator $[\alpha_{i}]$ of the chain group $C_{2}(X)$ (cf. \cite[Section 4.2]{GS}). 
	Then, we have for the first Chern class $c_{1}(X, J)$ of $(X, J)$, $$\langle c_{1}(X, J), [\alpha_{i}] \rangle = \rot (\alpha_{i}).$$

   	\begin{proof}[Proof of Theorem \ref{thm: main}]
    	
                	Fix $a_{0}$ in $\del D_{1}^{2}$ and $b_{0}$ in $\del D_{2}^{2} (a_{0})$.
               	At first, in order to construct a braided surface, 
            	we give elements $\beta_{1, N}, \beta_{2, N}, \beta_{3, i}, \beta_{4, i}, \beta_{5, i}$ of $B_{8(N-1)}$ for $i=1,2, \dots, N-1$ 
		as depicted in Figure \ref{fig: braids}. 
		 After this proof, we give explicit braid words of these braids.
		 Now we define a braided surface $S_{N}$ of degree $6N-4$ to be one with braid monodromy 
            	\begin{align}\label{braid monodromy}
            		(\beta_{1, N}, \beta_{2, N}, \beta_{3, 1}, \beta_{3, 2}, \dots, \beta_{3, N-1}, \beta_{4, 1}, \beta_{4, 2}, \dots, \beta_{4, N-1}, 
            		\beta_{5, 1}, \beta_{5, 2}, \dots, \beta_{5, N-1}).
            	\end{align}
	
		\begin{figure}[h]
            		\begin{center}
            			\includegraphics[width=450pt]{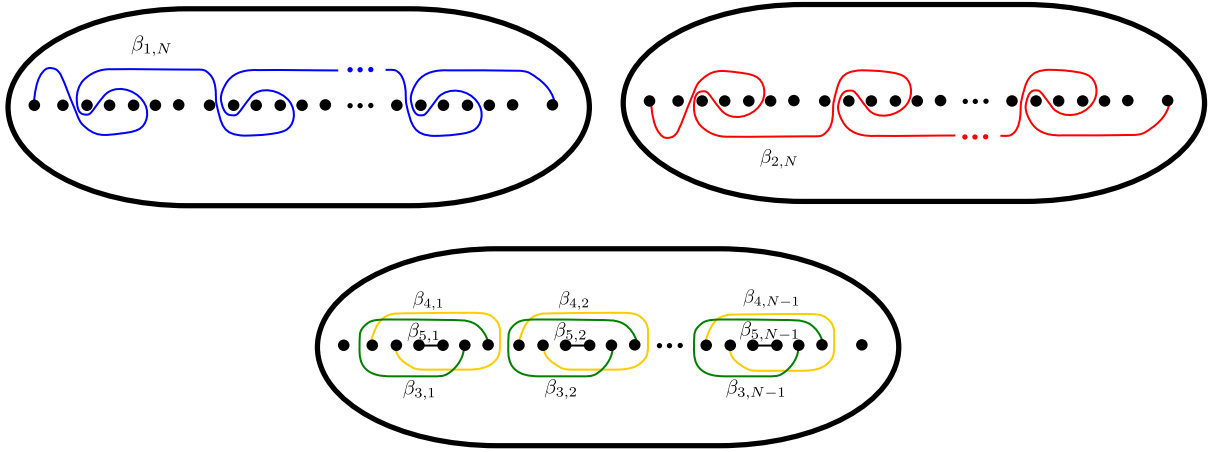}
            			\caption{Generating arcs of $\beta_{1, N}, \beta_{2, N}, \beta_{3, i}, \beta_{4, i}, \beta_{5, i} \in B_{8(N-1)}$ as elements of $\M_{0,1}^{m}$.}
            			\label{fig: braids}
            		\end{center}
            	\end{figure}

		In order to use Lemma \ref{lemma: lift}, 
		we need to construct appropriate simple branched covers of $D_{2}^{2}(a_{0})$.
		Define a simple branched covering 
		$q_{N,j}: \Sigma_{0,3N-1} \rightarrow D_{2}^{2}(a_{0})$ of degree $3N-1$ for each $j=1,2,\dots, N$ 
		as shown in Figure \ref{fig: covering}.
		After this proof, by using a covering monodromy, 
		we will describe this covering more explicitly.
		According to \cite[Lemma 3.2.3]{MM} for example, 
		we can check that each braid of the tuple (\ref{braid monodromy}) 
		is liftable with respect to each covering $q_{N,j}$ (see Figure \ref{fig: covering}). 
		It follows from Lemma \ref{lemma: lift} that for each covering $q_{N,j}$, 
		there exists a simple branched covering $p_{N, j}: X_{N,j} \rightarrow D^4$ branched along $S_{N}$ 
		such that $p_{N,j} | p_{N,j}^{-1}(D_{2}^{2}(a_{0})) = q_{N,j}$.
            	\begin{figure}[ht]
            		\begin{center}
            			\includegraphics[width=300pt]{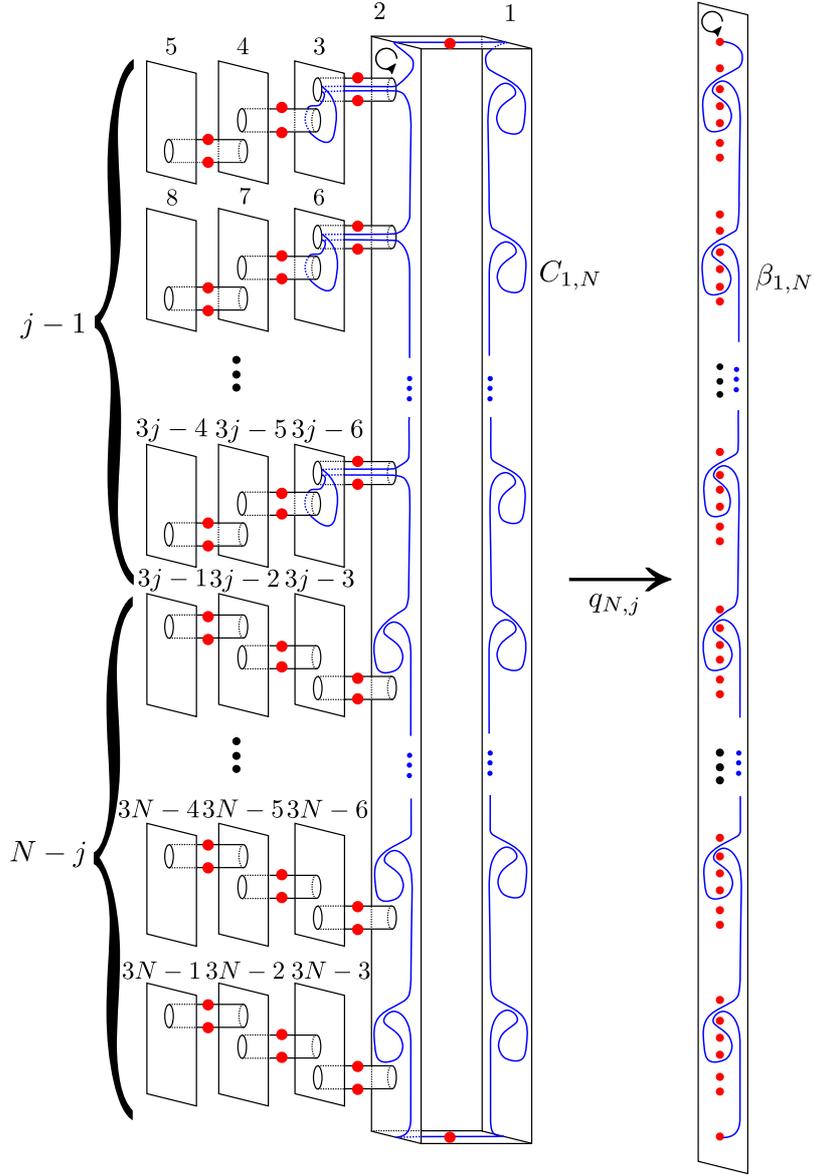}
            			\caption{Covering $q_{N,j}: \Sigma_{0,3N-1} \rightarrow D_{2}^{2}(a_{0})$. Each number represents an index of each sheet of the covering.
				The blue curve $C_{1,N}$ in the cover is the closed component of the lift of the arc generating the half-twist $\beta_{1,N}$}
            			\label{fig: covering}
            		\end{center}
            	\end{figure}

		We show that $X_{N,1}, X_{N,2}, \dots, X_{N, N}$ are mutually diffeomorphic.
		To do this, we will make use of stabilizations of open books and 
		the classification of Stein fillings of a contact manifold.
		The map $pr_{1} \circ p_{N, j}: X_{N, j} \rightarrow D_{1}^{2}$ is a PALF, and
		a monodromy of this PALF is the lift of the braid monodromy (\ref{braid monodromy}) of $S_{N}$ by $q_{N,j}$.
		Write 
		\[
			(t_{C_{1,N}^{j}}, t_{C_{2,N}^{j}}, t_{C_{3,1}^{j}}, t_{C_{3,2}^{j}}, \dots, t_{C_{3,N-1}^{j}}, t_{C_{4,1}^{j}}, t_{C_{4,2}^{j}}, \dots, t_{C_{4,N-1}^{j}}, 
			 t_{C_{5,1}^{j}}, t_{C_{5,2}^{j}}, \dots, t_{C_{5,N-1}^{j}})
		\] 
		for this monodromy of $pr_{1} \circ p_{N, j}$, 
		where each $C_{k,i}^{j}$ is the simple closed curve generating the right-handed Dehn twist 
		as the lift of $\beta_{k, i}$ by $q_{N,j}$ (see Figure \ref{fig: lift}).
		We obtain from this PALF an open book decomposition of $\del X_{N,1}$ 
		with page $\Sigma_{0, 3N-1}$ and monodromy
		\[
		 	\psi_{N, j}:= t_{C_{1,N}^{j}} t_{C_{2,N}^{j}} t_{C_{3,1}^{j}} t_{C_{3,2}^{j}} \cdots 
			t_{C_{3,N-1}^{j}} t_{C_{4,1}^{j}} t_{C_{4,2}^{j}} \cdots t_{C_{4,N-1}^{j}}  t_{C_{5,1}^{j}} t_{C_{5,2}^{j}} 
			\cdots t_{C_{5,N-1}^{j}} .
		\]
               	\begin{figure}[h]
            		\begin{center}
            			\includegraphics[width=450pt]{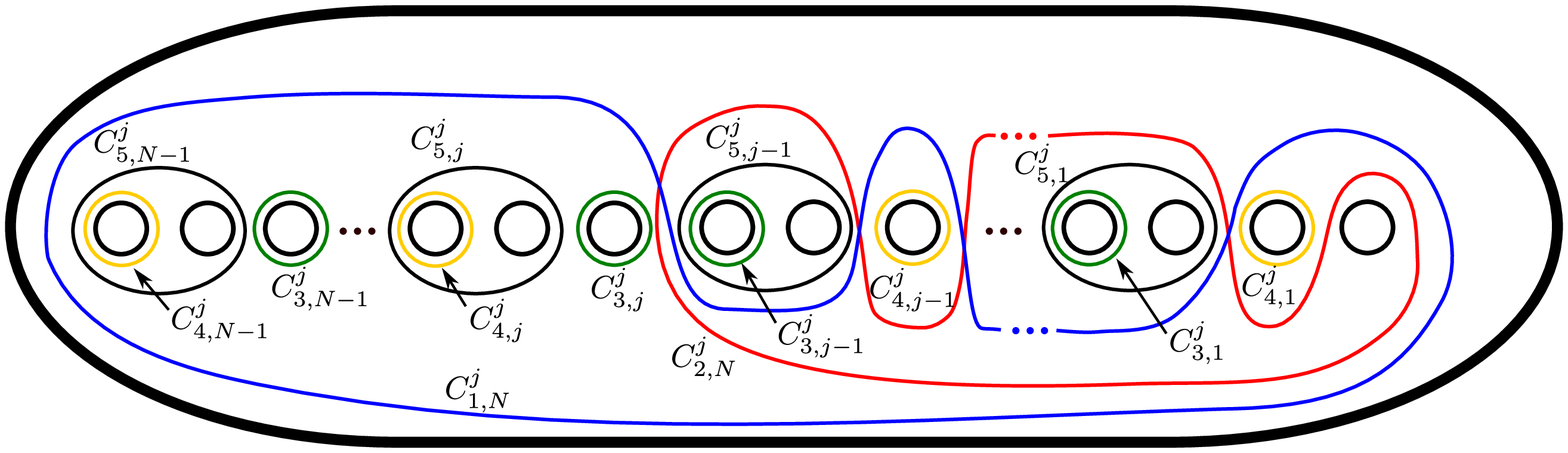}
            			\caption{Vanishing cycles of the PALF $pr_{1} \circ p_{N, j}: X_{N,j} \rightarrow D_{1}^{2}$.}
            			\label{fig: lift}
            		\end{center}
            	\end{figure}
	        	\begin{figure}[h]
            		\begin{center}
            			\includegraphics[width=450pt]{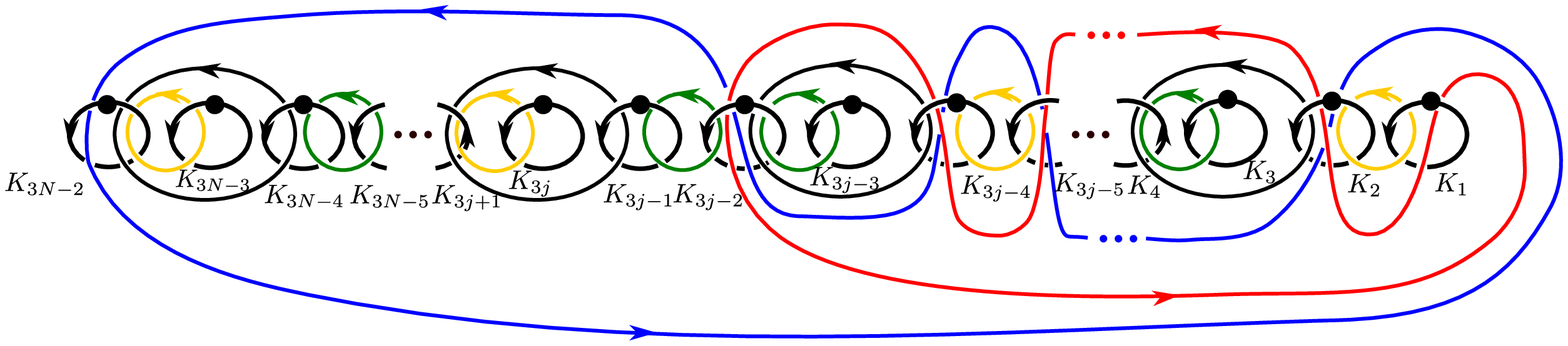}
            			\caption{Kirby diagram of $X_{N, j}$}
            			\label{fig: diagram}
            		\end{center}
		\end{figure}
		Stabilize the open book $(\Sigma_{0, 2N}, \varphi_{N, j})$ $N-1$ times as shown in Figure \ref{fig: how_to_stabilize}. 
		We can easily check that the resulting stabilization is $(\Sigma_{0, 3N-1}, \psi_{N, j})$.
	        	\begin{figure}[h]
            		\begin{center}
            			\includegraphics[width=350pt]{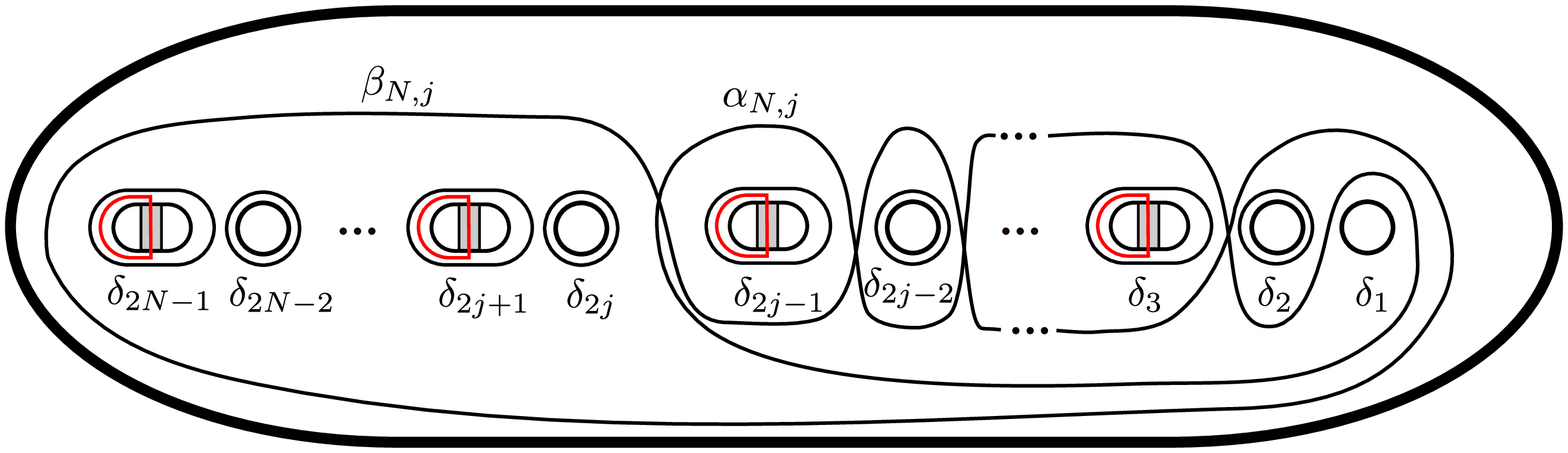}
            			\caption{Stabilization of the open book decomposition $(\Sigma_{0,2N}, \varphi_{N,j})$.
				Each shaded band (resp. red curve) represents an added band (resp. curve generating added Dehn twists) by this stabilizing. }
            			\label{fig: how_to_stabilize}
            		\end{center}
		\end{figure}
		Hence, each $\del X_{N, j}$ is diffeomorphic to $L(2N,1)$.
		Furthermore, since a positive stabilization of an open book supports the same contact structure 
		supported by the previous one, 
		$(\Sigma_{0,3N-1}, \psi_{N, j})$ supports the contact structure $\xi_{N,j}$.		
		Consider $X_{N,j}$ as a Stein filling of $(L(2N, 1), \xi_{N,j})$ and 
		apply to $X_{N,j}$ the classification of Stein fillings of $(L(2N, 1), \xi_{N,j})$.
		By \cite[Corollary 1.3]{PV}, $X_{N,j}$ is diffeomorphic to the disk bundle $X(S^2, -2N)$ over $S^2$ with Euler number $-2N$ if $(N,j) \neq (2,1)$; 
		otherwise either $X(S^2, -4)$ or the rational ball with Euler characteristic $1$. 
		In our case, 
		since $X_{2,1}$ admits the PALF $pr_{1} \circ p_{2,1}$ with fiber $\Sigma_{0,5}$ and $5$ critical points, 
		the Euler characteristic of $X_{2,1}$ is $2$. 
		Thus, every $X_{N, j}$ is diffeomorphic to $X(S^2, -2N)$. 
		
		To finish the proof, we compute the first Chern class of each $(X_{N,j}, J_{N,j})$, 
		where $J_{N, j}$ is the Stein structure associated to $p_{N,j}$.
		To begin with, we draw a Kirby diagram 
		of $X_{N,j}$ from the PALF structure to examine a generator of $H_{2}(X_{N,j}; \Z) \cong \Z$.
		Since the regular fiber of $pr_{1} \circ p_{N, j}$ is $\Sigma_{0,3n-1}$, 
		a Kirby diagram of $X_{N, j}$ has 
		$3N-2$ dotted $1$-handles $K_{1}, K_{2}, \dots, K_{3N-2}$. 
		We give orientations to the dotted circles of these $1$-handles, 
		and the attaching circles $C_{1, N}^{j}, C_{2, N}^{j}, C_{3,1}^{j}, \dots, C_{5, N-1}^{j}$ of these $2$-handles as in Figure \ref{fig: diagram}. 
		These $[K_{1}], \dots, [K_{3N-2}]$ and $[C_{1, N}^{j}], \dots, [C_{5, N-1}^{j}]$ freely generate the chain group $C_{1}(X)$ and $C_{2}(X)$ respectively 
		(cf. \cite[Section 4.2]{GS} and \cite[Section 2.3]{OS}).
		We claim that, for each $j = 1, 2, \dots, N$, 
		\[
			Z_{N, j}:= [C_{1, N}^{j}] - [C_{2, N}^{j}] - \Sigma_{i=j}^{N-1}[C_{3,i}^{j}] - \Sigma_{i=1}^{j-1}[C_{4,i}^{j}] + \Sigma_{i=1}^{j-1}[C_{5,i}^{j}] - \Sigma_{i=j}^{N-1}[C_{5,i}^{j}]
		\]
		is a generator of $H_{2}(X_{N, j}; \Z)$. 
		Let $\del : C_{2}(X) \rightarrow C_{1}(X)$ denote the boundary operator on the two chain groups $C_{1}(X)$ and $C_{2}(X)$ defined by 
		$\del ([C]) := \Sigma _{i=1} ^{3N-2} lk(C, K_{i})[K_{i}]$ on the generators and extended linearly. 
		Here, $lk(C, K_{i})$ is the linking number of $C$ and $K_{i}$.
		We have 
		\begin{align*}
			& \del([C_{1, N}^{j}] - [C_{2, N}^{j}])\\
			= & \{ ([K_{2}]-[K_{4}]) + ([K_{5}]-[K_{7}]) + \cdots + ([K_{3j-4}] -[K_{3j-2}]) + [K_{3N-2}] \}  \\ 
			 &- \{ ([K_{1}] - [K_{2}]) + ([K_{4}]- [K_{5}]) + \cdots + ([K_{3j-5}] - [K_{3j-4}]) + [K_{3j-2}] \} \\
			 = & -[K_{1}] + 2\{ ([K_{2}]-[K_{4}])+([K_{5}]-[K_{7}]) + \cdots + ([K_{3j-4}] - [K_{3j-2}]) \}+ [K_{3N-2}],
		\end{align*}
		\begin{align*}
			& \del(-\Sigma_{i=j}^{N-1}[C_{3,i}^{j}] - \Sigma_{i=1}^{j-1}[C_{4,i}^{j}]) \\
			= &- \{  (-[K_{3j-2}] + [K_{3j-1}]) + (-[K_{3j+1}]+ [K_{3j+2}])+ \cdots + (-[K_{3N-5}]+ [K_{3N-4}])\} \\
			& - \{ (-[K_{1}]+[K_{2}]) +(-[K_{4}]+ [K_{5}]) + \cdots + (-[K_{3j-5}]+ [K_{3j-4}])\} \\
			= & [K_{1}] - \{ ([K_{2}]- [K_{4}]) + ([K_{5}]-[K_{7}]) + \cdots + ([K_{3j-4}] - [K_{3j-2}]) \} \\
			& - \{ ([K_{3j-1}]-[K_{3j+1}]) + ([K_{3j+2}]-[K_{3j+4}]) +\cdots + ([K_{3N-7}]- [K_{3N-5}])\} - [K_{3N-4}], 
		\end{align*}
		 and 
		 \begin{align*}
		 	& \del ( \Sigma_{i=1}^{j-1}[C_{5,i}^{j}] - \Sigma_{i=j}^{N-1}[C_{5,i}^{j}]) \\
			= & -\{ ([K_{2}]-[K_{4}]) + ([K_{5}]-[K_{7}]) + \cdots + ([K_{3j-4}] - [K_{3j-2}])\} \\
			& +\{ ([K_{3j-1}]-[K_{3j+1}]) + ([K_{3j+2}]-[K_{3j+4}]) + \cdots\\ 
			&  + ([K_{3N-7}] - [K_{3N-5}]) + ([K_{3N-4}] - [K_{3N-2}]) \} .
		 \end{align*}
		  Therefore, 
		  $$\del(Z_{N, j}) = \del ([C_{1, N}^{j}] - [C_{2, N}^{j}] - \Sigma_{i=j}^{N-1}[C_{3,i}^{j}] - \Sigma_{i=1}^{j-1}[C_{4,i}^{j}] + \Sigma_{i=1}^{j-1}[C_{5,i}^{j}] - \Sigma_{i=j}^{N-1}[C_{5,i}^{j}]) = 0,$$ 
		  and 
		  $Z_{N, j}$ is an element of $\textrm{Ker}\, \del$.
		  Since $\textrm{Ker}\, \del \cong H_{2}(X_{N, j}; \Z) \cong \Z$ is abelian and 
		  the coefficient of each term in $Z_{N,j} = 
		  [C_{1, N}^{j}] - [C_{2, N}^{j}] - \Sigma_{i=j}^{N-1}[C_{3,i}^{j}] - \Sigma_{i=1}^{j-1}[C_{4,i}^{j}] + \Sigma_{i=1}^{j-1}[C_{5,i}^{j}] - \Sigma_{i=j}^{N-1}[C_{5,i}^{j}]$ 
		  is either $1$ or $-1$, it is a generator of $\textrm{Ker}\, \del$.
		  Now, we can compute $c_{1}(X_{N, j}, J_{N,j})$. 
		  Strictly speaking, 
		  $c_{1}(X_{N,j}, J_{N, j})$ can evaluate on the generator of $H_{2}(X_{N, j}; \Z)$ as follows: 
		  \begin{align*}
		  	& \langle c_{1}(X_{N,j}, J_{N, j}), Z_{N, j} \rangle \\
			= & \langle c_{1}(X_{N,j}, J_{N, j}), 
			[C_{1, N}^{j}] - [C_{2, N}^{j}] - \Sigma_{i=j}^{N-1}[C_{3,i}^{j}] - \Sigma_{i=1}^{j-1}[C_{4,i}^{j}] + \Sigma_{i=1}^{j-1}[C_{5,i}^{j}] - \Sigma_{i=j}^{N-1}[C_{5,i}^{j}] \rangle \\
			= & \rot(C_{1, N}^{j}) - \rot(C_{2, N}^{j}) - \Sigma_{i=j}^{N}\rot(C_{3,i}^{j}) - \Sigma_{i=1}^{j-1}\rot(C_{4,i}^{j}) + \Sigma_{i=1}^{j-1}\rot(C_{5,i}^{j}) - \Sigma_{i=j}^{N}\rot(C_{5,i}^{j})\\
			= & 0 - 0 - \Sigma_{i=j}^{N-1} (+1) - \Sigma_{i=1}^{j-1} (+1) + \Sigma_{i=1}^{j-1} (+1) - \Sigma_{i=j}^{N-1} (+1) \\
			= & -2(N-j)
		\end{align*}
		Therefore, $c_{1}(X_{N,j}, J_{N,j}) \neq c_{1}(X_{N, j'}, J_{N, j'})$ if $j \neq j'$, and the Stein structures $J_{N,1}, J_{N,2}, \dots, J_{N, N}$ are mutually not homotopic.
		\end{proof}
		 
		In the following, we give explicit descriptions of the braids $\beta_{1,N}, \beta_{2,N}, \beta_{3,i}, \beta_{4,i}, \beta_{5,i}$ and 
		covering $q_{N,j}:\Sigma_{0,3N-1}\rightarrow D^{2}_{2}(a_{0})$ in the above proof.
		These are not essential in this article, hence 
		 the reader can skip them if he or she likes.
		
		We give first explicit braid words of the above braids. 
		For convenience, 
		set $$\beta_{1}^{\beta_{2}} := \beta_{2}^{-1} \beta_{1} \beta_{2}$$ 
		for $\beta_{1}, \beta_{2} \in B_{m}$, and 
		$$\tau_{i,j} := \sigma_{i}^{\sigma_{i+1} \sigma_{i+2} \cdots \sigma_{j-1}},\ 
		\taub_{i,j} := \sigma_{i}^{\sigma_{i+1}^{-1} \sigma_{i+2}^{-1} \cdots \sigma_{j-1}^{-1}} \in B_{m},$$
		where $i<j$ (see Figure \ref{fig: tau}).
		Obviously, if $j= i+1$, both $\tau_{i,j}$ and $\taub_{i,j}$ are $\sigma_{i}$.	
		\begin{figure}[b]
            		\begin{center}
       	     			\includegraphics[width=230pt]{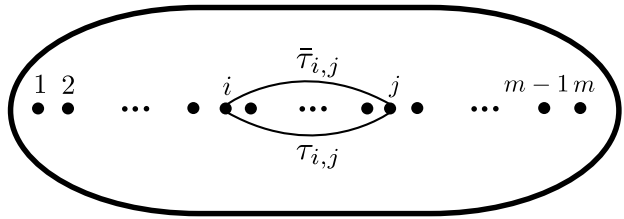}
            			\caption{Generating curves of $\tau_{i,j}$ and $\taub_{i,j}$, where we identify $B_{m}$ with $\M_{0,1}^{m}$. }
          			\label{fig: tau}
            		\end{center}
		\end{figure}
		Define 
		$$ T_{1,N} := {\Pi_{j=1}^{N-1}\{ \tau^{-1}_{2+6(j-1), 5+6(j-1)} \taub^{-1}_{5+6(j-1), 8+6(j-1)}\} \Pi_{j=1}^{N-1}\{ \sigma_{4+6(j-1)} \sigma_{3+6(j-1)}\} }, $$
		$$ \overline{T}_{1,N} := {\Pi_{j=1}^{N-1}\{ \taub_{2+6(j-1), 5+6(j-1)} \tau_{5+6(j-1), 8+6(j-1)}\} \Pi_{j=1}^{N-1}\{ \sigma^{-1}_{4+6(j-1)} \sigma^{-1}_{3+6(j-1)}\} }.$$
		Then, we have 
		$$
            		\beta_{1, N} :=   \sigma_{1}^{T_{1,N}},\  \beta_{2, N} :=   \sigma_{1}^{\overline{T}_{1,N}},
		$$
		$$
            		\beta_{3, i} :=  \tau_{2+6(i-1), 6+6(i-1)}^{\taub_{2+6(i-1), 7+6(i-1)}^{-1}}, \ 
            		\beta_{4, i} :=  \tau_{3+6(i-1), 7+6(i-1)}^{\taub_{2+6(i-1), 7+(6i-1)}}, \ 
            		\beta_{5, i} :=  \tau_{6i-2,6i-1} = \sigma_{6i-2}.
		$$

		Next, we describe the covering $q_{N,j}: \Sigma_{0,3N-1} \rightarrow D_{2}^{2}(a_{0})$ by using a covering monodromy.
		Let $(\gamma'_{1}, \gamma'_{2}, \dots, \gamma'_{6N-4})$ 
		be the standard Hurwitz system for $(D_{2}^{2}(a_{0}) \cap S_{N}, (a_{0},b_{0}) )$.
		By observing Figure \ref{fig: covering}, we can check that a covering monodromy 
		$\rho_{N,1}: \pi_{1}(D_{2}^{2}(a_{0})-S_{N}, (a_{0} , b_{0})) \rightarrow \Es_{3N-1}$ of the covering $q_{N,j}$ is a homomorphism defined by 
		\[
            		\rho_{N,j} (\gamma'_{i})= 
            			\begin{cases}
            				(1\ 2) &  (i=1, 6N-4), \\
            				(2\ 3k) &  (i=2+6(k-1), 3+6(k-1); k=1,2,\dots, j-1), \\
            				(3k\ 3k+1) &  (i=4+6(k-1), 5+6(k-1); k=1,2,\dots, j-1), \\
					(3k+1\ 3k+2) & (i=6+6(k-1), 7+6(k-1); k=1,2,\dots, j-1),\\
					(3\ell+1\ 3\ell +2) &  (i=2+6(\ell-1), 3+6(\ell-1); \ell=j, j+1, \dots, N-1), \\
            				(3\ell\ 3\ell+1) &  (i=4+6(\ell-1), 5+6(\ell-1); \ell=j,j+1,\dots, N-1), \\
					(2\ 3\ell) & (i=6+6(\ell-1), 7+6(\ell-1); \ell=j,j+1,\dots, N-1).
            			\end{cases}
            	\]
		

	\begin{remark}
            	In the above proof, the case of $N=2$ is crucial, so 
            	we explain how the author found the braided surface $S_{2}$. 
            	First, he fixed two different branched coverings $q_{2,1}$ and $q_{2,2}$ and 
		considered liftable braids with respect to both coverings.
            	He observed how corresponding lifts change if we change $q_{2,1}$ into $q_{2,2}$, and 
            	he chose some braids among them to obtain the braided surface $S_{2}$. 
            	Finally, drawing Kirby diagrams of the two corresponding covers branched along $S_{2}$, 
            	he checked whether these covers satisfied the conditions of our theorem.
            	Hence, his construction is very ad hoc. 
            	As far as he knows, there is no systematic construction of such a braided surface.
	\end{remark}

       	\begin{proof}[Proof of Corollary \ref{cor: transverse}]
                	Note that the boundary of a given braided surface $S$ is contained in $\del D_{1}^{2} \times D_{2}^{2}$, 
		and it is the closure of a braid. 
                	Letting $U$ be the core of $D_{1}^{2} \times \del D_{2}^{2}$, 
                	we obtain from the product structure on $\del D_{1}^{2} \times D_{2}^{2}$, 
		an open book decomposition of $ S^3 \approx \del D^{4} \approx \del (D_{1}^{2} \times D_{2}^{2})$ whose page is a disk and binding is $U$.
    		This open book supports the standard contact structure $\xi_{std}$ on $S^3$.
                	Thus, we can regard $\del S$ as a transverse link in $(S^3, \xi_{std})$. 
    	
    		Let $L_{N}$ be the boundary of $S_{N}$ in the proof of Theorem \ref{thm: main}.
    		By the previous argument, $L_{N}$ can be seen as a transverse link in $(S^{3}, \xi_{std})$. 
    		Set $M_{N, j} := \del X_{N, j}$. 
    		$M_{N,1}, M_{N, 2}, \dots , M_{N,N}$ are mutually diffeomorphic to $L(2N, 1)$. 
    		Here, $p_{N, j}: X_{N,j} \rightarrow D^4$ restricts on the boundary $M_{N, j}$ to 
		the simple branched cover of $S^3$ branched along $L_{N}$.
		By the proof of Theorem \ref{thm: main}, 
		the associated open book $(\Sigma_{0, 3N-1}, \psi_{N, j})$ to the covering is 
		adapted to $(L(2N, 1), \xi_{N, j})$. 
		By the classification of tight contact structures on $L(2N, 1)$ in \cite[Theorem 2.1]{Ho}, 
		$\xi_{N,1}, \xi_{N,2}, \dots, \xi_{N, N}$ are mutually not isotopic.    		
    	  \end{proof}

\bibliographystyle{amsplain}

\end{document}